\documentclass[letterpaper,11pt]{amsart}
\usepackage[margin=1.2in]{geometry}
\usepackage{amsmath,amsthm,amssymb}
\usepackage{xspace,xcolor}
\usepackage[breaklinks,colorlinks,citecolor=teal,linkcolor=teal,urlcolor=teal,pagebackref,hyperindex]{hyperref}
\usepackage[alphabetic]{amsrefs}
\usepackage[all]{xy}
\usepackage[english]{babel}
\usepackage{enumitem}
\usepackage{tikz,xcolor}
\usepackage{tikz-cd}
\usepackage{mathrsfs}
\usepackage{soul}
\setstcolor{red}
\usepackage{color}

\DeclareMathOperator{\DDB}{\underline{\Omega}}

\DeclareMathOperator{\IO}{I\hspace{0.07em}\underline{\Omega}}
\DeclareMathOperator{\IC}{IC}
\newcommand{\kdot}{{{\,\begin{picture}(1,1)(-1,-2)\circle*{2}\end{picture}\,}}}

\theoremstyle{plain}
\newtheorem{thm}{Theorem}[section]

\newtheorem{lem}[thm]{Lemma}
\newtheorem{prop}[thm]{Proposition}

\newtheorem{conj}[thm]{Conjecture}

\theoremstyle{definition}
\newtheorem{defi}[thm]{Definition}

\newtheorem{setting}[thm]{Setting}

\newtheorem{egs}[thm]{Examples}

\theoremstyle{remark}
\newtheorem{rmk}[thm]{Remark}

\def\Z{{\mathbf Z}}
\def\Q{{\mathbf Q}}

\def\C{{\mathbf C}}

\def\A{{\mathbf A}}

\def\cD{\mathcal{D}}

\def\cF{\mathcal{F}}
\def\cG{\mathcal{G}}
\def\cH{\mathcal{H}}

\def\cJ{\mathcal{J}}

\def\cL{\mathcal{L}}

\def\cO{\mathcal{O}}

\def\cT{\mathcal{T}}

\def\J{\mathcal{J}}

\def\frm{\mathfrak{m}}

\def\m{\mathfrak{m}} 

\def\.{\cdot}
\def\^{\widehat}

\def\({\left(}
\def\){\right)}

\renewcommand{\and}{ \ \ \text{ and } \ \ }

\DeclareMathOperator{\Spec} {Spec}

\DeclareMathOperator{\Hom} {Hom}

\DeclareMathOperator{\Exc} {Exc}

\begin{document}

\author[B.~Dirks]{Bradley Dirks}

\address{Department of Mathematics, Stony Brook University, Stony Brook, NY 11794-3651, USA}

\email{bradley.dirks@stonybrook.edu}

\author[J.~Witaszek]{Jakub Witaszek}

\address{Northwestern University, Department of Mathematics, Lunt Hall, 2033 Sheridan Road, Evanston IL, 60208, USA}

\email{jakub.witaszek@northwestern.edu}

\begin{abstract} 
We give Hodge-theoretic characterizations of characteristic-zero varieties of open $F$-nilpotent and open weakly $F$-nilpotent type, extending a result of Srinivas and Takagi beyond isolated singularities. We prove one direction unconditionally and the converse assuming the weak ordinarity conjecture.
\end{abstract}

\title{$F$-nilpotence and Hodge filtrations beyond isolated singularities}
\maketitle

\tableofcontents

\section{Introduction}

Over the complex numbers, some important classes of singularities, such as \emph{rational} and \emph{Du Bois}, are defined via birational geometry or Hodge-theoretic conditions. In positive characteristic, singularities are often studied through the action of Frobenius, giving rise to notions such as \emph{$F$-rational}, \emph{$F$-injective}, and \emph{$F$-nilpotent}.

Comparing classes of singularities in characteristic zero and positive characteristic is an important problem in the field. By \cites{Smith97,MS97,Hara98}, a characteristic zero singularity is rational if and only if its reduction modulo $p \gg 0$ is $F$-rational. Conjecturally, a characteristic zero singularity is Du Bois if and only if its reduction modulo  infinitely many primes $p$ is $F$-injective. The implication from right to left is a theorem due to Schwede \cite{Schwede09}, while the converse is proven in \cite{BST} assuming the weak ordinarity conjecture (Conjecture \ref{conj:weakordinarity}). 

The goal of this article is to discuss analogous results for $F$-nilpotent singularities, introduced in \cite{BB05} under the name ``close to $F$-rational singularities'' and also called ``$F$-rational up to nilpotents'' in \cite{BBLSZ}. 
\begin{defi}[{Proposition \ref{prop:EquivFNilpDef}}]
\label{defi:FnilpIntro} Let $(R,\frm)$ be a local $d$-equidimensional Noetherian reduced ring of characteristic $p>0$. We say that 
\begin{enumerate}
    \item $R$ is \emph{weakly $F$-nilpotent}, if  $H^i_\frm(R)$ is Frobenius nilpotent\footnote{Namely, for every $\alpha \in H^i_\m(R)$ there exists $e>0$ such that $F^e(\alpha)=0.$} for all $i < d$, 
    \item $R$ is \emph{$F$-nilpotent} if, in addition, every Frobenius $R$-submodule of $H^{d}_{\frm}(R)$, which is proper modulo every minimal prime, is also Frobenius nilpotent (see Proposition \ref{prop:EquivFNilpDef} for the formal statement).
\end{enumerate}

\end{defi}
The idea of $F$-nilpotency is that it measures the difference between $F$-rational and $F$-injective singularities.
\begin{defi} \label{defi:introDenseOpen}
 For a finite type locally equidimensional scheme $X$ over a field $k$ of characteristic zero, we say that $X$ is of \emph{open $F$-nilpotent type} (resp.\ \emph{open weakly $F$-nilpotent type}) if there exists a flat model $X_A$ over a finitely generated $\Z$-subalgebra $A$ of $k$ and a Zariski-open set of closed points $S \subseteq {\rm Spec}\, A$ such that $X_s$ is $F$-nilpotent (resp.\ weakly $F$-nilpotent) for every $s \in S$.
\end{defi}

It is natural to ask which characteristic zero singularities correspond to $F$-nilpotent ones. This question was addressed by Srinivas and Takagi \cite{Srinivas-Takagi} in the case of isolated singularities. More precisely, they conjectured that an $n$-dimensional normal isolated singularity $x \in X$ over $\mathbb{C}$ is of open $F$-nilpotent type if and only if  the zeroth graded piece of the Hodge filtration of local cohomology vanishes:
\[
{\rm Gr}^0_F H^i_{\{x\}}(X_{\rm an}, \mathbb{C}) = 0 \quad \text{for all } i \in \Z.
\]
Moreover, they established the implication from right to left with no extra assumptions, and proved the converse assuming the weak ordinarity conjecture. 

We propose the following conjecture, which extends Srinivas and Takagi's statement beyond the isolated case and provides a more conceptual interpretation of their condition.

\begin{conj} \label{conj:mainintro}
Let $X$ be an equidimensional variety defined over a field of characteristic zero and let $f \colon Y \to X$ be a resolution of singularities. Then the following statements hold.

{\setlength{\leftmargini}{2em} \begin{enumerate}
\setlength\itemsep{0.1em}
\item $X$ is of open $F$-nilpotent type if and only if $f^* \colon \DDB^0_X \to Rf_*\cO_Y$ is an isomorphism.
\item $X$ is of open weakly $F$-nilpotent type if and only if $\DDB^0_X$ is Cohen--Macaulay. 
\end{enumerate}}
\end{conj}
\noindent The map in the above conjecture is constructed by way of the functoriality of $\DDB^0$, namely as a  composition:
\begin{equation} \label{eq:keymap}
f^* \colon \DDB^0_X \to Rf_*\DDB^0_Y \xleftarrow{\cong} Rf_*\cO_Y.
\end{equation}
The condition that this map is an isomorphism measures the difference between rational and Du Bois singularities (see Remark \ref{rmk:MeasureDiffRatDuBois}).
\begin{rmk} One can check that $Rf_*\cO_Y$ is Cohen--Macaulay by Grauert-Riemenschneider vanishing, and so the condition that $f^* \colon \DDB^0_X \to Rf_*\cO_Y$ is an isomorphism implies that $\DDB^0_X$ is Cohen--Macaulay.
\end{rmk}

Our main theorem addresses one direction of this conjecture unconditionally, while the converse is proven assuming the weak ordinarity conjecture as in the work of Srinivas and Takagi.

\begin{thm} \label{thm-main}
Let $X$ be an equidimensional variety defined over a field of characteristic zero and let $f \colon Y \to X$ be a resolution of singularities.  Then:

{\setlength{\leftmargini}{2em} \begin{enumerate}
\setlength\itemsep{0.1em}
\item if $f^* \colon \underline{\Omega}_X^0 \to Rf_*\cO_Y$ is an isomorphism, then $X$ is of open $F$-nilpotent type, 

\item if $\DDB^0_X$ is Cohen--Macaulay, then $X$ is of open weakly $F$-nilpotent type. 
\end{enumerate}}
\noindent Moreover, the converses of the above statements hold assuming the weak ordinarity conjecture.
\end{thm}

We note that, as in the work of Srinivas and Takagi, the weak ordinarity conjecture is used for the implication from characteristic $p>0$ to  characteristic zero. This is the opposite of what happens in the study of $F$-injectivity. Moreover, it is not enough to assume that $X$ is of dense $F$-nilpotent type to deduce that $f^* \colon \underline{\Omega}_X^0 \to Rf_*\cO_Y$ is an isomorphism. Indeed, there are examples when this map is not an isomorphism and $X$ is of dense but not open $F$-nilpotent type (see Examples \ref{eg-Examples}). 

 \begin{rmk} \label{rmkInterpretation} The conditions for $X$ to be of (weak) $F$-nilpotent type in Conjecture~\ref{conj:mainintro} are not incidental, but part of a broader picture. Namely, the map $f^* \colon \underline{\Omega}_X^0 \to Rf_*\cO_Y$ is an isomorphism exactly when the Hodge rational homology level satisfies ${\rm HRH}(X) \geq 0$, while the condition that $\DDB^0_X$ is Cohen–Macaulay is equivalent to $c(X) \geq 0$. The first invariant was introduced in \cites{ParkPopa, DOR1} (called ``property $(D_0)$'' in the former) to measure Hodge-theoretic properties of singularities. The second was introduced in \cite{CDOIsolated} as a natural weakening of the complete intersection condition. The comparison becomes particularly clear for closed embeddings $X \subseteq Y$ into a smooth variety $Y$:
\begin{itemize} \item In characteristic $0$: 
\begin{align*}
\hspace{7.9em} c(X) \geq 0 \iff&\, F_0 \cH^q_X(\cO_Y) = 0 \text{ for all } q > {\rm codim}_Y(X),\\
{\rm HRH}(X) \geq 0 \iff&\,   {\footnotesize(1)}\ F_0 \cH^q_X(\cO_Y) = 0 \text{ for all } q > {\rm codim}_Y(X), \text{ and } \\
&\ {\footnotesize(2)}\  F_0 {\rm IC}_X^H(-c) = F_0 \cH^{{\rm codim}_Y(X)}_X(\cO_Y).
\end{align*}
\item In characteristic $p$:
\begin{align*}
X \text{ is weakly }F\text{-nilpotent} \iff& \cH^q_X(\cO_Y) = 0 \text{ for all } q > {\rm codim}_Y(X),\\
X \text{ is }F\text{-nilpotent} \iff&\,   {\footnotesize(1)}\  \cH^q_X(\cO_Y) = 0 \text{ for all } q > {\rm codim}_Y(X), \text{ and } \\
&\,  {\footnotesize(2)}\ \cL(X,Y) = \cH^{{\rm codim}_Y(X)}_X(\cO_Y).
\end{align*}
\end{itemize}
\noindent We refer to  Remark~\ref{rmk:BB}, as well as Section~\ref{ss:HRHl} and Remark~\ref{rmk:cHRHHodgeWeight} for the definitions and a more detailed discussion. 

In particular, our main theorem, contingent upon the weak ordinarity conjecture, shows that for $X$ of characteristic zero, there exists an open subset $S \subseteq \Spec A$ of closed points (see the notation in Definition \ref{defi:introDenseOpen}) such that
\[
F_0 \cH^q_X(\cO_Y) = 0 \text{ for all } q > {\rm codim}_Y(X)
\quad\Longleftrightarrow\quad
\cH^q_{X_s}(\cO_{Y_s}) = 0 \text{ for all } q > {\rm codim}_{Y_s}(X_s),
\]
together with an analogous statement for the intersection complexes ${\rm IC}_X^H$ and $\cL(X,Y)$.
\end{rmk}

Lastly, for completeness, we state the weak ordinarity conjecture from \cite[Conjecture 1.1]{MS11} and \cite[Conjecture 1.1]{BST}. Note that in the work of Takagi and Srinivas, a slightly weaker conjecture is assumed (equivalent when considered for varieties of all dimensions, but strictly weaker in fixed dimension), which allows them to prove their theorem unconditionally in dimension three. We do not pursue this here.
\begin{conj}[Weak ordinarity conjecture] \label{conj:weakordinarity}
Let $Y$ be a $d$-dimensional smooth projective variety over a characteristic zero field $k$. Given a model of $Y$ over a finitely generated $\Z$-subalgebra $A$ of $k$, there exists a Zariski-dense set of closed points $S \subseteq {\rm Spec}(A)$ such that Frobenius acts bijectively on $H^d(Y_s, \cO_{Y_s})$ for every $s \in S$.
\end{conj}

In fact, this conjecture is equivalent to Du Bois singularities being of $F$-injective type (see \cites{BST,MS11}).

\subsection{Acknowledgments}
The authors thank Bhargav Bhatt, Tatsuro Kawakami, Mircea Musta{\c{t}}{\u{a}}, and Karl Schwede for valuable conversations related to the content of the paper. Dirks was supported by the NSF grant DMS-1926686 and the NSF-MSPRF grant DMS-2303070. Witaszek was supported by NSF research grants DMS-2101897 and DMS-2401360.
\section{Preliminaries} 
\subsection{Notation and terminology} \label{ss:notation}
We refer to \cite[Tag 08XG]{stacks-project} or  \cite[Subsection 2.7]{KTTWYY1} for a quick review of local and Matlis dualities which will be extensively used in this article. Moreover, in this paper, we use the following notation:
\begin{itemize}
\item A \textit{variety} is a reduced  separated scheme of finite type over a field.

\item We say that a proper morphism $f \colon Y \to X$ of Noetherian schemes is \emph{birational} if it is an isomorphism over an open dense subset $U \subseteq X$. Given a proper birational morphism $f\colon Y\to X$, we denote by $\Exc(f)$ the \emph{reduced exceptional locus}, namely $f^{-1}(X \setminus U)^{\rm red}$ where $U$ is the largest open dense subset over which $f$ is an isomorphism.

\item Given a reduced Noetherian scheme $X$, a proper birational morphism $\pi \colon Y \to X$ is called \emph{a log resolution of $X$} if $Y$ is regular and $\mathrm{Exc}(\pi)$ is a simple normal crossing divisor. Assume in addition that $X$ is excellent. Then a log resolution is called \emph{strong} if it is an isomorphism over the regular locus of $X$. A log resolution of a pair $(X,Z)$ for a closed subset $Z \subseteq X$ containing the singular locus of $X$ is called \emph{strong} if it is an isomorphism over the complement of $Z$.

\item For a ring $R$ we denote by $R^\circ$ the complement of the union of all minimal primes. When $R$ is reduced, this is precisely the set of
non-zero-divisors of $R$. We follow the convention that $0 \not \in R^\circ$. 
\item Given a Noetherian scheme $X$, we say that $K \in D^b_{\rm coh}(X)$ is \emph{Cohen--Macaulay} if $H^i_x(K) = 0$ for all points $x \in X$ and $0 \leq i < \dim \cO_{X,x}$. This condition is sometimes called \emph{maximal} Cohen--Macaulay in the literature.
\item Given a positive characteristic ring $R$, its \emph{perfection} is the ring defined by the following formula:
\[
R_{\rm perf} = \varinjlim \left(R \xrightarrow{F} R \xrightarrow{F}  \ldots \right). 
\]
\item We shall implicitly use throughout this paper that $F$-finite Noetherian rings in characteristic $p>0$ are excellent.
\item Throughout this paper, whenever convenient, we implicitly use the recently proven fact that Noetherian $F$-finite schemes admit a dualising complex \cite{bhatt2026ffiniteschemesdualizingcomplex}. We refer to \cite[Tag 0A85]{stacks-project} for the definition of, and a discussion on, dualising complexes $\omega^\bullet_X$. When $X$ is equidimensional of dimension $d$, we normalize $\omega^\bullet_X$ so that it lies in degrees $[-d,0]$ and define the \emph{dualising sheaf} by $\omega_X := \cH^{-d}(\omega^\bullet_X)$.
\item We let $\DDB^i_X \in D^b_{\rm coh}(X)$ denote the $i$-th Du Bois complex of a variety $X$ over a field of characteristic $0$. When $X$ is smooth, one has $\DDB^i_X=\Omega^i_X[0]$. In general, it is given by
\[
\DDB^i_X := R\epsilon_{\kdot,*}\Omega^i_{X_\kdot},
\]
where $\epsilon_\kdot \colon X_{\kdot}\to X$ is a hyperresolution (see \cite[Chapter~7.3]{PetersSteenbrink} for details). Equivalently, $\DDB^i_X$ is the derived  sheafification of $\Omega^i$ on the $h$-site over $X$. In this paper, we denote $\cH^0(\DDB^i_X)$ by $\Omega^i_{X,h}$. 
\end{itemize}

In this article we shall extensively study the Deligne-Du Bois complex $\DDB^0_X$. A key result about it that we shall often refer to in this article is the \emph{injectivity theorem}. We formulate it here as a dual statement. 
\begin{thm}[{\cite[Theorem 3.3]{KS}}] \label{thm:injectivity}
Let $X$ be a reduced  separated finite type scheme defined over a field of characteristic zero. Then the natural map
\[
H^i_x(\cO_X) \to H^i_x(\DDB^0_X)
\]
is surjective for every point $x \in X$ and integer $i \in \Z$.
\end{thm}
In particular, this shows that
\begin{equation} \label{eq:DDBCM}
\text{ $X$ is Cohen--Macaulay} \implies \text{ $\DDB^0_X$ is Cohen--Macaulay.}
\end{equation}

Moreover, with the same assumptions, we have another injectivity-type theorem:
\begin{equation} \label{eq:injectivity2}
f^* \colon H^i_x(\cO_X) \to H^i_x(Rf_*\cO_Y)
\end{equation}
is surjective for every point $x \in X$, for every $i\in \Z$, and for every resolution of singularities $f \colon Y \to X$. This result, however, is rather easy. Indeed, $Rf_*\cO_Y$ is Cohen--Macaulay by local duality and Grauert-Riemenschneider vanishing, and so it remains to show that
\begin{equation} \label{eq:injectivity2ford}
f^* \colon H^d_x(\cO_X) \to H^d_x(Rf_*\cO_Y)
\end{equation}
is surjective for $d := \dim \cO_{X,x}$. Using the long exact sequence of local cohomology associated to the short exact sequence
\[
0 \to \cO_X \to \nu_* \cO_{X^{\rm n}} \to \cF \to 0,
\]
with $\nu \colon X^{\rm n} \to X$ being the normalization and $\cF$ being the natural cokernel, we reduce to the case of $X$ being normal. Then local duality translates (\ref{eq:injectivity2ford}) into  the injectivity of the natural map ${\rm Tr}_f \colon f_*\omega_Y \hookrightarrow \omega_X$  (cf.\ \cite[Tag 0AWK]{stacks-project}).

In particular, using the factorisation $f^* \colon \cO_X \to \DDB^0_X \to Rf_*\cO_Y$ we get that
\begin{equation} \label{eq:injectivity3}
f^* \colon H^i_x(\DDB^0_X) \to H^i_x(Rf_*\cO_Y)
\end{equation}
is surjective for every point $x \in X$ and every integer $i \in \Z$.

\subsection{Birational singularities}
We briefly review classical singularities in characteristic zero.
\begin{defi} \label{def:DuBoisRational}
Let $(R,\frm)$ be a local $d$-dimensional reduced ring essentially of finite type over a field of characteristic zero. We say that $R$ has \emph{rational singularities} if for every (equivalently, for some) resolution of singularities $f \colon Y \to X := \Spec R$, the natural map
\[
f^* \colon \cO_X \to Rf_*\cO_Y
\]
is an isomorphism. We say that $R$ has \emph{Du Bois singularities} if the natural map
\[
\cO_X \to \DDB^0_X
\]
is an isomorphism.
\end{defi}
\noindent  Note that rational singularities are automatically normal. On the other hand, Du Bois singularities need not be normal or even integral.

When $X := \Spec R$ has isolated normal singularities, the condition that $X$ is Du Bois is equivalent by \cite{SteenbrinkDB}*{Pg.\ 533} to the natural map
\[
R^if_*\cO_Y \to H^i(E, \cO_E)
\]
being an isomorphism for every $i > 0$, where $f \colon Y \to X$ is a log resolution of singularities with exceptional locus $E$ and which is an isomorphism over the regular locus of $X$ (see (\ref{eq:HPS})). Moreover, by Grothendieck duality and Grauert-Riemenschneider vanishing, $R$ has rational singularities if and only if $f_*\omega_Y = \omega_X$ and $X$ is Cohen--Macaulay.
\begin{rmk} \label{rmk:DuBoisLC}
Set $X := \Spec R$. By Theorem \ref{thm:injectivity}, the ring $R$ is Du Bois if and only if
\[
H^i_\m(\cO_X) \to H^i_\m(\DDB^0_X)
\]
is injective for every integer $i \in \Z$. Indeed, as can be readily deduced from local duality, a map in $D^b_{\rm coh}(\cO_X)$ is an isomorphism if it induces an isomorphism on all local cohomologies at maximal ideals. An analogous statement holds for rational singularities: $R$ is rational if and only if
\begin{equation} \label{eq:remRational}
H^i_\m(\cO_X) \to H^i_\m(Rf_*\cO_Y)
\end{equation}
is injective for every integer $i \in \Z$ (see (\ref{eq:injectivity2})). This observation allows for a better comparison of the notions of rationality and Du Bois with their positive characteristic counterparts: $F$-rational and $F$-injective.
\end{rmk}

\begin{rmk} \label{rmk:MeasureDiffRatDuBois}
We see that the condition from Conjecture \ref{conj:mainintro}, that the map
\[
f^*\colon \DDB^0_X \to Rf_*\cO_Y
\]
is an isomorphism, measures the difference between having rational and Du Bois singularities:
\[
\text{ $X$ is rational } \iff \text{$X$ is Du Bois and $f^*\colon \DDB^0_X \to Rf_*\cO_Y$ is an isomorphism.}
\]
The implication $(\implies)$ is non-obvious and follows from Theorem \ref{thm:injectivity} and Remark \ref{rmk:DuBoisLC}.
\end{rmk}

The above naturally leads us to the construction of a sheaf measuring rationality. For the sake of future discussion, we define it in the more general setting   of arbitrary characteristic. We restrict to the local case as this is all that will be needed in our article.
\begin{defi} \label{def:GR}
Let $(R,\m)$ be a local $d$-equidimensional reduced ring essentially of finite type over a field. We define the \emph{Grauert-Riemenschneider sheaf} as
\[
\J(\omega_R) := \bigcap_{f \colon Y \to X} {\rm im}\left({\rm Tr}_f \colon f_*\omega_Y \to \omega_R\right) 
\]
where the intersection is taken over all projective birational morphisms $f \colon Y \to X := \Spec R$ and ${\rm Tr}_f$ denotes the trace map.
\end{defi}

\begin{rmk}
Given the torsion-freeness of $\omega_Y$, we can also write this in a more familiar form \[\J(\omega_R) = \bigcap_{f \colon Y \to X} f_*\omega_Y.\]
\end{rmk}

One can check that if a resolution of singularities $f \colon Y \to X$ exists, then $\J(\omega_R) = {\rm im}({\rm Tr}_f \colon f_*\omega_Y \to \omega_R)$. Indeed, if $g \colon Z \to Y$ is any projective birational morphism, then the regularity of $Y$, the torsion-freeness of $\omega_Z$, and the existence of the pullback map \[
\omega_Y = \Omega^d_Y \xrightarrow{g^*} g_*\Omega^d_Z \to g_*(\Omega^d_Z)^{**} = g_*\omega_Z
\]
imply that $g_*\omega_Z = \omega_Y$. The proof then follows by taking a common birational refinement of any map with this resolution.

\begin{rmk}
We reinterpret the Grauert-Riemenschneider sheaf using local cohomology as this will be needed later on. With notation as in Definition \ref{def:GR}, we define
\[
0^{\rm GR}_{H^d_\m(R)} := \bigcup_{f \colon Y \to X} {\rm Ker}\Big(f^* \colon H^d_{\m}(R) \to H^d_{\m}(Rf_*\cO_Y)\Big)
\]
where the union is taken over all projective birational morphisms $f \colon Y \to X$. As before, if a resolution of singularities $f \colon Y \to X$ exists, then one can check that (cf.\ (\ref{eq:GRMD})):
\[
0^{\rm GR}_{H^d_\m(R)} = {\rm Ker}\Big(f^* \colon H^d_{\m}(R) \to H^d_{\m}(Rf_*\cO_Y)\Big).
\]
Finally, $0^{\rm GR}_{H^d_\m(R)}$ is related to the Grauert-Riemenschneider sheaf via the following formula, provided a resolution exists:
\begin{equation} \label{eq:GRMD}
\J(\omega_R)^{\wedge} \cong \Big(H^d_\m(R)/0^{\rm GR}_{H^d_\m(R)} \Big)^\vee
\end{equation}
where $(-)^{\wedge}$ is the completion at $\m$ and  $(-)^{\vee}$ denotes Matlis duality. 
\end{rmk}

\subsection{F-singularities}

In this subsection, we review $F$-rational and $F$-injective  singularities. We refer to \cite{SchwedeSmith_FrobeniusSingularities} for more details.

\begin{defi}[{\cite[Definition 1.2]{Srinivas-Takagi}}]
Let $(R,\frm)$ be a local $d$-dimensional $F$-finite Noetherian ring of characteristic $p>0$. We say that $R$ is \emph{$F$-injective} if the Frobenius action on local cohomology
\[
F \colon H^i_\frm(R) \to H^i_\frm(F_*R)
\]
is injective for every $i \in \Z$. We say that $R$ is \emph{$F$-rational} if $R$ is Cohen--Macaulay and for every non-zero divisor $c \in R^\circ$ the composite map
\begin{equation*}
H^d_\frm(R) \to H^d_\frm(F^e_*R) \xrightarrow{\cdot F^e_* c } H^d_\frm(F^e_*R)
\end{equation*}
is injective for some integer $e>0$ (equivalently, for every $e \gg 0$). We say that a positive characteristic Noetherian ring is $F$-injective (resp.\ $F$-rational) if it is so after localisation at every maximal ideal.
\end{defi}

One can check that $F$-injective singularities are automatically reduced (see \cite[Corollary 3.5]{DM24}) and that $F$-rationality implies normality (see \cite[Lemma 2.34]{bstAlt} or  \cite[Corollary 2.14 and Proposition 3.9]{TW18}).

\begin{rmk}[{cf.\ \cite[Proposition 3.19]{TW18}}] \label{rmk:dualDefFSing}
For Cohen--Macaulay rings $R$ as above, using Matlis duality we see that:
\begin{align*}
\text{ $R$ is $F$-injective } &\iff T \colon F_*\omega_R \to \omega_R \text{ is surjective, }\\
\text{ $R$ is $F$-rational } &\iff \forall_{c \in R^\circ} \exists_{e}\ T^{e,c} \colon F^e_*\omega_R \to \omega_R \text{ is surjective},\\
&\iff \forall_{c \in R^\circ} \exists_{e_0} \forall_{e \geq e_0}\ T^{e,c} \colon F^e_*\omega_R \to \omega_R \text{ is surjective},\\
&\iff \forall_{c \in R^\circ} \forall_{e_0} \exists_{e \geq e_0}\ T^{e,c} \colon F^e_*\omega_R \to \omega_R \text{ is surjective},
\end{align*}
where $T^{e,c} \colon F^e_*\omega_R \to \omega_R$ is the Frobenius trace map $T^e \colon F^e_*\omega_R \to \omega_R$ pre-composed with multiplication by $F^e_*c$. The choice of quantifiers in the third equivalent definition of $F$-rationality is most natural in the context of parameter test ideals and tight closure below. 
\end{rmk}

\subsection{Test ideals} \label{ss:testideals}
In this subsection, following the brief recollection from \cite[Section 2.2]{KTTWYY3},  we review the classical theory of test ideals and refer to standard sources \cites{schwedetucker12,Tak21,SchwedeSmith_FrobeniusSingularities} for more details. We shall work under the following assumptions.
\begin{setting} \label{setting:TestIdealsReview}
In what follows, $R$ is a  $d$-equidimensional reduced Noetherian $F$-finite ring of characteristic $p>0$. 
\end{setting}
\noindent 
Although \cite[Section 2.2]{KTTWYY3} is written for complete local
domains, the definitions and statements recalled below hold in the
present reduced, locally equidimensional setting; we indicate general
references at the relevant points.

\begin{defi}[{cf.\ \cite[Definition 7.1 and Lemma 7.6]{epstein2026closureoperationsinducedresolutions}}] \label{defi:idealspos} Under the assumptions of Setting \ref{setting:TestIdealsReview}, we define the  $R$-submodule $\sigma(\omega_R) \subseteq \omega_R$ as follows:
\begin{align*}
\sigma(\omega_R) := \bigcap_{e>0} {\rm Im}(T^e \colon F^e_*\omega_R \to \omega_R).
\end{align*}
Similarly, we define the \emph{parameter test submodule} $\tau(\omega_R) \subseteq \omega_R$ by the formula:
\begin{align*}\tau(\omega_R) &:= \bigcap_{c \in R^\circ} \bigcap_{e_0>0} \sum_{e \geq e_0} {\rm Im}(T^{e,c} \colon F^e_*\omega_R \to \omega_R).
\end{align*}
\end{defi}

By Remark \ref{rmk:dualDefFSing}, we see that $\sigma(\omega_R)$ and $\tau(\omega_R)$ measure $F$-injectivity and $F$-rationality of Cohen--Macaulay rings $R$, respectively: 
\begin{align*}
\text{ $R$ is $F$-injective } &\iff \sigma(\omega_R)=\omega_R, \text{ and }\\
\text{ $R$ is $F$-rational } &\iff \tau(\omega_R)=\omega_R.
\end{align*}

\begin{rmk}[{cf.\ \cite[Definition 1.3]{Srinivas-Takagi}}]
We relate the above ideals to the study of local cohomology. Let $(R,\m)$ be a local ring satisfying
Setting \ref{setting:TestIdealsReview}. We define the \emph{Frobenius closure of $0$ in $H^d_\frm(R)$} as follows:
\begin{align*}
0^F_{H^d_\m(R)} :=&\ \bigcup_{e>0} {\rm Ker}\Big(F^e \colon H^d_{\m}(R) \to H^d_{\m}(F^e_*R)\Big) 
=\  {\rm Ker}\Big(H^d_{\m}(R) \to H^d_{\m}(R_{\rm perf})\Big).
\end{align*}
Similarly, we define the \emph{tight closure of $0$ in $H^d_\frm(R)$} via the formula:
\begin{align*}
0^*_{H^d_\m(R)} &:= \bigcup_{\substack{c \in R^{\circ},\\ e_0 \in \Z_{>0}}} \bigcap_{e \geq e_0} K^{e,c}, \quad \text{ where }\\
K^{e,c} &:= {\rm Ker}\Big(H^d_{\m}(R) \xrightarrow{F^e} H^d_{\m}(F^e_*R) \xrightarrow{\cdot F^e_*c} H^d_{\m}(F^e_*R) \Big) 
\end{align*}
Equivalently, for $z \in H^d_\m(R)$, 
    we have that $z \in 0^*_{H^d_\m(R)}$ if and only if 
    there exists $c \in R^{\circ}$ and an integer $e_0>0$ such that 
    $cF^e(z)=0$ for every integer $e \geq e_0$.\\

By definition $0^F_{H^d_\m(R)} \subseteq 0^*_{H^d_\m(R)}$.  Moreover, these closure operations determine the submodules from Definition \ref{defi:idealspos}:
\begin{align} \label{eq:tauMD}
\sigma(\omega_R)^{\wedge} &\cong \Big(H^d_\m(R)/0^F_{H^d_\m(R)} \Big)^\vee, \text{ and } \\
\tau(\omega_R)^{\wedge} &\cong \Big(H^d_\m(R)/0^*_{H^d_\m(R)} \Big)^\vee, \nonumber
\end{align}
where $(-)^{\wedge}$ is the completion at $\m$ and  $(-)^{\vee}$ denotes Matlis duality. The first statement is immediate by Matlis duality and stabilisation (cf.\ Remark \ref{rmk:stabilisation}). The second statement follows from, for instance,  \cite[Lemma 7.6]{epstein2026closureoperationsinducedresolutions} (cf.\ \cite[Definition 2.3(3) and Definition 2.9(3)]{KTTWYY3}).
\end{rmk} 

\begin{rmk} \label{rmk:analyticallyIrred}
In fact, $\tau(\omega_R)$ is the smallest non-zero submodule of $\omega_R$ which is stable under $T \colon F_*\omega_R \to \omega_R$ and non-zero on every irreducible component of $\Spec R$ of maximal dimension  (see \cite[Definition 2.33]{bstAlt}, cf.\ \cite{Smith97}, \cite[Definition 2.9(6)]{KTTWYY3}).

In particular, the tight closure $0^*_{H^d_\m(R)}$ is  the largest $R$-submodule of $H^d_\m(R)$ which is Frobenius stable and does not surject onto any $H^d_\m(R/\mathfrak{q}_i)$ for minimal primes $\mathfrak{q}_1, \ldots, \mathfrak{q}_m$ of $R$. From this, it is easy to deduce that  under the assumptions of Setting \ref{setting:TestIdealsReview}, we have the following equivalence (\cite[Theorem 2.6]{Smith97}):
\begin{enumerate}
    \item $R$ is $F$-rational,
    \item $R$ is Cohen--Macaulay and $H^d_\frm(R)$ admits no proper non-zero Frobenius stable $R$-submodule.
\end{enumerate}
\end{rmk}

\begin{rmk}[Stabilisation] \label{rmk:stabilisation}
By \cite[Proposition 2.14]{BB11Finiteness}, we have that \[
\sigma(\omega_R)= {\rm Im}(T^e \colon F^e_*\omega_R \to \omega_R)
\]
for every $e \gg 0$.  Analogously, one can show that
\begin{align*}\tau(\omega_R) &=  \sum_{e \geq 0} {\rm Im}(T^{e,c} \colon F^e_*\omega_R \to \omega_R)
\end{align*}
for a particular element $c \in R^\circ$ called a \emph{big test element} (see \cite[Lemma 4.3]{bstAlt}, cf.\ \cite[Definition 2.9]{KTTWYY3}). Explicitly, one may take any $c \in R^\circ$ such that the localisation $R_c$ is regular, and then replace $c$ by a sufficiently large power. Moreover, by Noetherianity, one can check that
\[
\tau(\omega_R) = {\rm Im}(T^{e,c} \colon F^e_*\omega_R \to \omega_R)
\]
for some $e>0$ sufficiently large, provided that $c$ is replaced by its square (cf.\ \cite[Definition 2.9(7)]{KTTWYY3}).

In particular, by this remark, the formation of both $\sigma(\omega_R)$ and $\tau(\omega_R)$ commutes with localisation. Moreover, it commutes with completion as well.
\end{rmk}

Finally, we compare $\tau(\omega_R)$ with the Grauert-Riemenschneider sheaf from characteristic zero  and refer to \cite{SchwedeSmith_FrobeniusSingularities} for more details and \cite[Proposition 2.12]{epstein2026closureoperationsinducedresolutions} for the statement with the current assumptions.

\begin{thm}[{\cites{Smith00,Hara05}}] \label{thm:taureduction}
Let $X$ be an equidimensional affine variety defined over a field $k$ of characteristic zero and let $X_A$ be a flat model over a finitely generated $\Z$-subalgebra $A$ of $k$. Then there exists a Zariski-open subset of closed points $S \subseteq \Spec A$ such that for every $s \in S$:
\[
\tau(\omega_{X_s}) = \J(\omega_{X_s}).
\]
\end{thm}
In particular, by (\ref{eq:GRMD}) and (\ref{eq:tauMD})  we get that for $s \in S$ and all closed points $\m \in \Spec X_s$:
\begin{equation} \label{eq:0FstarEqualRed}
0^{\rm GR}_{H^d_\m(\cO_{X_s})} = 0^*_{H^d_\m(\cO_{X_s})}. 
\end{equation}

\subsection{$F$-nilpotent singularities} \label{ss:Fnilp}
In this subsection we discuss $F$-nilpotent and weakly $F$-nilpotent singularities which are the main focus of this article. 
\begin{defi} \label{defi:Fnilp} Let $(R,\frm)$ be a local $d$-equidimensional reduced Noetherian $F$-finite ring of characteristic $p>0$. We say that $R$ is 
\begin{enumerate}
    \item \emph{weakly $F$-nilpotent} if  $H^i_\frm(R)$ is Frobenius nilpotent for every $0 \leq i < d$, and 
    \item \emph{$F$-nilpotent} if it is weakly $F$-nilpotent and $0^*_{H^d_\m(R)}$ is Frobenius nilpotent.
\end{enumerate}
We say that a positive characteristic reduced Noetherian locally equidimensional $F$-finite ring $R$ is \emph{$F$-nilpotent} (resp.\ \emph{weakly $F$-nilpotent}) if it is so after localisation at every maximal ideal.
\end{defi}

\begin{rmk} \label{rmk-PerfCM} It is immediate from the definition that if $(R,\frm)$ is a Cohen--Macaulay local ring, then $R$ is weakly $F$-nilpotent. A similar implication holds in characteristic $0$, but it is not so immediately clear; see (\ref{eq:DDBCM}).    
\end{rmk}

By definition, we get that
\[
\text{ $R$ is $F$-rational } \iff \text{$R$ is $F$-injective and $F$-nilpotent.}
\]

For the convenience of the reader, we review the following standard result.

\begin{prop} \label{prop:EquivFNilpDef}
Let $(R,\m)$ be a local $d$-equidimensional reduced $F$-finite Noetherian ring of characteristic $p>0$. Assume that $R$ is weakly $F$-nilpotent. Then the following statements are equivalent:
\begin{enumerate} \setlength\itemsep{0.2em}
\item $R$ is $F$-nilpotent,
\item $0^F_{H^d_\m(R)} = 0^*_{H^d_\m(R)}$,
\item $\sigma(\omega_R) = \tau(\omega_R)$.
\item every Frobenius stable $R$-submodule $M$ of $H^{d}_{\frm}(R)$ which does not map surjectively onto any $H^d_\frm(R/\mathfrak{q}_i)$ for minimal primes $\mathfrak{q}_1, \ldots, \mathfrak{q}_m$ of $R$, is also Frobenius nilpotent.
\end{enumerate}
\end{prop}
\noindent In particular, Definition \ref{defi:Fnilp} agrees with Definition \ref{defi:FnilpIntro} from the introduction.
\begin{proof}
The equivalence $(1) \iff (2)$ is definitional and $(2) \iff (3)$ follows from (\ref{eq:tauMD}). The equivalence with (4) is immediate by Remark \ref{rmk:analyticallyIrred}.
\end{proof}

We conclude with a few remarks which, although not used in this paper, shed further light on the meaning of $F$-nilpotence.

\begin{rmk} \label{rmk:FinjPerf}
It is easy to see that $R$ is $F$-injective if and only if the natural map
\[
H^i_\frm(R) \to H^i_\frm(R_{\rm perf})
\]
is injective for every $i \in \Z$. Since $R_{\rm perf}$ acts as the positive characteristic analogue of $\DDB^0_R$, this gives a direct comparison between the definition of $F$-injectivity and Du Bois singularities as in Remark \ref{rmk:DuBoisLC}. Similarly -- though this time the result is a deep theorem related to $+$-closure -- one can check that a domain $R$ is $F$-rational if and only if
\[
H^i_\m(R) \to H^i_\m(R^+)
\]
is injective for every $i \in \Z$ (cf.\ \cites{Smith94, HH92}), where $R^+$ is the absolute integral closure of $R$, namely $R^+ = \bigcup_{R \subseteq S} S$ where the union is taken over all finite extensions $R \subseteq S$ contained within a fixed algebraic closure $\overline{{\rm Frac}(R)}$ of the fraction field of $R$. Intuitively, finite covers serve as replacements for birational maps in positive characteristic; in this vein, $R^+$ is the positive characteristic analogue of $Rf_*\cO_Y$ from Definition \ref{def:DuBoisRational} (see also Remark \ref{rmk:DuBoisLC}).

In the spirit of the above paragraph, the definition of $F$-nilpotence from Definition \ref{defi:Fnilp} can be checked to be equivalent to the injectivity of
\[
H^i_\m(R_{\rm perf}) \to H^i_\m(R^+)
\]
for every $i \in \Z$. This further emphasizes that $F$-nilpotence measures the difference between $F$-injectivity and $F$-rationality.
\end{rmk}

\begin{rmk} \label{rmk:BB} The notion of weak $F$-nilpotence is related to Hartshorne-Speiser's \emph{$F$-depth} \cite{HartshorneSpeiser}. Indeed, \cite{BlickleBondu} shows that weak $F$-nilpotence is equivalent, when $R = S/I$ is a quotient of a regular local ring $S$ by an ideal $I$, to the vanishing $\cH^j_I(S) = 0$ for $j > {\rm ht}(I)$. By \cite{HartshorneSpeiser}*{Thm. 2.5(b)}, this is equivalent to $R$ having maximal $F$-depth, equal to $\dim R$.

In the same work, Blickle--Bondu defined a notion of \emph{close to $F$-rational singularities}, which is equivalent to $F$-nilpotent singularities. If $R = S/I$ where $S$ is a regular local ring, then $R$ has close to $F$-rational singularities if and only if the composition
\[ \cL(S,I) \to \cH^{{\rm ht}(I)}_I(S) \to R\Gamma_I(S)[{\rm ht}(I)]\]is a quasi-isomorphism, where $\cL(S,I)$ is Blickle's intersection homology $F$-module. This characterization matches the characteristic $0$ version in Remark \ref{rmk:cHRHHodgeWeight} below.

The above result can also be formulated intrinsically on $X := \Spec R$. Namely, it is shown in \cite{BBLSZ} that a $d$-dimensional Noetherian domain $R$ is weakly $F$-nilpotent if and only if $\mathbb{F}_p[d]$ is perverse as an \'etale constructible $\mathbb{F}_p$-sheaf. Moreover, $R$ is $F$-nilpotent if and only if $\mathbb{F}_p[d]$ coincides with the intersection complex.
\end{rmk}

Let  $X$ be a variety defined over a field of characteristic zero. The key property of $\DDB^i_X$ that we shall use -- and indeed, a defining one -- is the following. Given a blow-up square,
\begin{equation} \label{eq:blow-up-squareE}
\begin{tikzcd}
E \ar[hook]{r} \ar{d} & Y \ar{d}{f} \\
Z \ar[hook]{r} & X.
\end{tikzcd}
\end{equation}
namely, $f \colon Y \to X$ is a projective birational map which is an isomorphism over the complement of the closed subvariety $Z \hookrightarrow X$ with $E = f^{-1}(Z)^{\rm red}$, we obtain a homotopy pullback square: 
\begin{equation} \label{eq:HPS}
\begin{tikzcd}
Rf_*\DDB^{i}_{E}   & \ar{l} Rf_*\DDB^{i}_{Y}  \\
\DDB^{i}_{Z}  \ar{u} & \ar{l}  \ar{u}[swap]{f^*} \DDB^{i}_X.
\end{tikzcd}
\end{equation}
In particular, there exists an exact triangle:
\[
\DDB^{i}_X \to Rf_*\DDB^{i}_{Y} \oplus \DDB^{i}_{Z} \to  Rf_*\DDB^{i}_{E} \xrightarrow{+1}.
\]
Further, when $E$ is simple normal crossing, we have that $\DDB^i_E \cong \Omega^i_E/ {\rm tor}$ (\cite{DuBois}*{Example 4.7}). 

\begin{rmk} \label{rmk:SchwedeFormula}
Given an affine variety $X$ over a field $k$ of characteristic zero together with an embedding $u \colon X \hookrightarrow W$ into a smooth variety $W$ we can use the above remark to obtain an easy formula for $\DDB^0_X$ due to Saito and Schwede \cite{SchwedeDB}. Namely, let $f \colon Y \to W$ be a log resolution of $(W,X)$ which is an isomorphism over $W \setminus X$ with $E := f^{-1}(X)^{\rm red}$. Then we get a homotopy pullback square:
\begin{equation*}
\begin{tikzcd}
Rf_*\cO_{E}   & \ar{l} Rf_*\cO_{Y}  \\
\DDB^{0}_{X}  \ar{u}{f^*} & \ar{l}  \ar{u}[swap]{f^*} \cO_W.
\end{tikzcd}
\end{equation*}
Since $f^* \colon \cO_W \to Rf_*\cO_Y$ is an isomorphism, we thus get that the natural map
\begin{equation} \label{eq:Schwede}
f^* \colon \DDB^{0}_{X} \to Rf_*\cO_E \text{ is an isomorphism.}
\end{equation}
\end{rmk}
\vspace{1em}
Next, we move to the case when $X$ is a variety defined over a field of positive characteristic. In this case, it turns out that by a result of Gabber the derived sheafification of the functor $\Omega^0$ on the $h$-site of $X$ yields nothing but the perfection $\cO_{X, {\rm perf}}$ (see \cite[Theorem 3.3]{BST}). In particular, given a blow-up square as in (\ref{eq:blow-up-squareE}), we get the induced homotopy pullback square:
\begin{equation} \label{eq:HPSPerf}
\begin{tikzcd}
Rf_*\cO_{E, {\rm perf}}   & \ar{l} Rf_*\cO_{Y, {\rm perf}}   \\
\cO_{Z, {\rm perf}}   \ar{u} & \ar{l}  \ar{u}[swap]{f^*} \cO_{X, {\rm perf}} ,
\end{tikzcd}
\end{equation}
and so the exact triangle:
\begin{equation} \label{eq:exact-triangle-Eperf}
\cO_{X, {\rm perf}}  \to Rf_*\cO_{Y, {\rm perf}}  \oplus \cO_{Z, {\rm perf}}  \to  Rf_*\cO_{E, {\rm perf}}  \xrightarrow{+1}.
\end{equation}
\begin{rmk}[Embedded resolution]
With the same assumptions and notation as in Remark \ref{rmk:SchwedeFormula} but with $k$ replaced by a characteristic $p>0$ field we immediately get that:
\begin{equation} \label{eq:SchwedeCharP}
f^* \colon \cO_{X,{\rm perf}} \to Rf_*\cO_{E, {\rm perf}} \text{ is an isomorphism.}
\end{equation}
\end{rmk}

\subsection{Hodge rational homology level} \label{ss:HRHl}
To place the conditions from Conjecture \ref{conj:mainintro} in a broader context, we review various Hodge theoretic conditions from the viewpoint of higher singularities. The content of this subsection will not be used in the proofs of the main theorems.

Let $X$ be a reduced equidimensional connected variety over $\C$. As all notions we consider in this subsection will be local, we may assume that $X$ is affine and pick an embedding $i \colon X \hookrightarrow W$ into a smooth ambient variety $W$. We refer to \cite{SchnellMHM} for an overview of the theory of mixed Hodge modules on a smooth variety (see also \cite[Section 2.4]{KW}). By a mixed Hodge module on $X$, we shall mean a mixed Hodge module on $W$ which is supported on $X$. We denote by ${\rm IC}^H_X$ the Hodge module intersection complex. Its underlying complex of constructible sheaves is the intersection complex ${\rm IC}_X \in D^b_{\rm cons}(X, \Q)$.

In what follows we shall utilize the functor
\[
{\rm Gr}^F_\bullet {\rm DR} \colon D^b_{\rm MHM}(X) \to D^b_{\rm coh, gr}(X)
\]
from the derived category of mixed Hodge modules on $X$ to the graded derived category of coherent sheaves on $X$. We refer to \cite[Proposition 2.17]{KW} for a recap of its properties.

It turns out that for a $d$-equidimensional reduced variety $X$ over $\C$, one has that 
\[
\DDB^i_X \cong {\rm Gr}^F_{-i} {\rm DR}(\Q^H_X[d])[-d+i].
\]
Building on \cite{KebekusSchnell}, Popa and Park introduced the following definition.
\begin{defi}[{\cite[Definition 3.3]{PPLefschetz}}]
  Let $X$ be an equidimensional connected variety of dimension $d$ over $\C$. We define the \emph{$i$th intersection Du Bois complex} by the formula:
  \[
  \IO^i_X := {\rm Gr}^F_{-i}{\rm DR}(\IC^H_X)[-d+i].
  \]
\end{defi}
By applying ${\rm Gr}^F_{\bullet}{\rm DR}(-)$ to the natural map $\Q^H_X[d] \to \IC^H_X$ we get induced morphisms
\begin{equation} \label{eq:QToIC}
\DDB^i_X \to \IO^i_X.
\end{equation}
\begin{rmk}

We say a few words on why (\ref{eq:QToIC}) agrees with map (\ref{eq:keymap}) from the introduction. Given a resolution of singularities $f\colon Y \to X$, the pullback morphism $\Q_X^H[d] \to f_* \Q_Y^H[d]$ factors through the natural map to ${\rm IC}_X^H$. By \cite{SaitoMHC}*{Proof of Theorem 5.4}, if we apply ${\rm Gr}^F_0 {\rm DR}(-)$ to this composition
\[
f^* \colon \Q_X^H[d] \to {\rm IC}_X^H \to f_* \Q_Y^H[d]
\]
we get
\[ \underline{\Omega}_X^0 \to \IO^0_X \xrightarrow{\cong} Rf_*\cO_Y,\]
agreeing with the natural morphism $f^*$ from (\ref{eq:keymap}). Here, the rightmost map is an isomorphism by \cite{KebekusSchnell}*{Proposition 8.2}.
\end{rmk}

\begin{rmk} Let $X$ be a normal connected variety over $\C$. Following \cite{SVV}, we say that $X$ is \emph{pre-$m$-Du Bois} if and only if $\DDB^i_X \cong \cH^0(\DDB^i_X)$ for every $i \leq m$. Moreover, one says that $X$ is \emph{pre-$m$-rational} if and only if $\IO^i_X \cong \cH^0(\IO^i_X)$ for every $i \leq m$. Note that this is not the original definition as stated in \cite{SVV}, but it is equivalent thereto by \cites{ParkPopa,PSV,DOR1} (see \cite[Corollary 4.10]{KW}). The notion of Hodge rational homology level defined below will measure the difference between pre-$m$-Du Bois singularities and pre-$m$-rational singularities. We refer to \cites{MOPW,JKSY,MPLocCoh,FL,CDM, DOR1,CDOIsolated,SVV,PSV,PPFactorial,KW} for more detailed discussion on higher singularities. 
\end{rmk}

We say that a variety $X$ over $\C$ is a \emph{rational homology manifold} if $\IC_X \cong \Q_X[\dim X]$. In view of the following morphisms induced by Poincar\'{e} duality:
\begin{equation} \label{eq-PD} \Q_X[\dim X] \to {\rm IC}_X \to \mathbf D_X(\Q_X[\dim X]),\end{equation}
one is naturally led to a weakening of the rational homology manifold condition in terms of the Hodge filtration. This gives rise to the following definition, whose properties were elucidated in \cites{ParkPopa,DOR1}.

\begin{defi} Let $X$ be an equidimensional variety over $\C$. Then we say that the \emph{Hodge rational homology level} of $X$ is at least $k$, denoted ${\rm HRH}(X) \geq k$,
if and only if the maps 
$\underline{\Omega}_X^i \to \IO_X^i$ (see \eqref{eq:QToIC}) 
are quasi-isomorphisms for all $i\leq k$.

We have ${\rm HRH}(X) = + \infty$ if and only if $X$ is a rational homology manifold.
\end{defi}

The invariant ${\rm HRH}(X)$ has been studied in various examples \cites{DOR1,DOR2,CDOIsolated,ORS}: hypersurfaces, LCI varieties, cones over rational homology manifolds, determinantal varieties and secant varieties. Moreover, the condition ${\rm HRH}(X) \geq 0$ is precisely the difference between Du Bois singularities and rational singularities.

The following is a more subtle invariant, measuring the Hodge-theoretic failure of $X$ to behave like a local complete intersection variety from a local cohomological point of view.

\begin{defi}[{\cite{CDOIsolated}}] Let $X$ be an equidimensional variety over $\C$. Define
\[ c(X) := \sup \{k \geq -1 \mid {\rm depth}(\underline{\Omega}_X^i) \geq \dim X - i \text{ for all } 0 \leq i \leq k\}.\]
This depth condition is equivalent to the vanishing $H^j_x(\DDB_X^i) = 0$ for all closed points $x \in X$ and all   $i,j$ such that $0 \leq i \leq k$ and $i+j < \dim X$.
\end{defi}

We make a few observations.
\begin{enumerate}
\setlength\itemsep{0.3em}
    \item We have that $c(X) \geq 0$ if and only if $\underline{\Omega}_X^0$ is Cohen--Macaulay. Note that this condition is automatically satisfied when $X$ is Cohen--Macaulay thanks to (\ref{eq:DDBCM}).
    \item We have $c(X) = \infty$ if and only if $\Q_X[\dim X]$ is a perverse sheaf, as discovered in \cite{MPLocCoh}*{Corollary 12.6};  see also \cite{PSV}*{Theorem 4.1}.
    \item Since ${\rm depth}(\IO^i_X) \geq \dim X - i$ for all $i \geq 0$ (cf.\  \cite[Lemma 2.31]{KW}), we get
\[ c(X) \geq {\rm HRH}(X).\]
\end{enumerate}

\begin{rmk} \label{rmk:cHRHHodgeWeight}
The invariants ${\rm HRH}(X)$ and $c(X)$ can be expressed in terms of Hodge and weight filtrations as proven in \cite{DOR1}*{Theorem D(1)} and \cite{CDOIsolated}*{Theorem A}. Namely, let $X \subseteq Y$ be a closed embedding into a smooth variety $Y$ with $c = \dim Y - \dim X$. Then
\begin{align*} c(X) &= \sup \{k \mid F_k \cH^{c+j}_X(\cO_Y) = 0 \text{ for all } j > 0\}\\[0.3em]
{\rm HRH}(X) &= \min\left(c(X), \sup \{k \mid F_k W_{\dim Y +c}\cH^{c}_X(\cO_Y) = F_k \cH^c_X(\cO_Y)\}\right).
\end{align*}
Note that the underlying $\cD_Y$-module of $W_{\dim Y +c} \cH^c_X(\cO_Y)$ is the intersection homology $\cD$-module of \cite{BrylinskiKashiwara} (see \cite{SaitoMHM}*{(4.5.9)}). In this way, we see that the ${\rm HRH}(X)$ invariant is strongly related to Blickle--Bondu's condition from Remark \ref{rmk:BB}.
\end{rmk}

We can now rephrase Conjecture \ref{conj:mainintro} in the spirit of Remark \ref{rmkInterpretation} as follows:
{\setlength{\leftmargini}{2em} \begin{enumerate}
\setlength\itemsep{0.1em}
\item $X$ is of open $F$-nilpotent type if and only if ${\rm HRH}(X) \geq 0$.
\item $X$ is of open weakly $F$-nilpotent type if and only if $c(X) \geq 0$. 
\end{enumerate}}
This reformulation naturally raises the question of whether there exist positive characteristic analogues of the conditions ${\rm HRH}(X) \geq k$ and $c(X) \geq k$. We will address this in a subsequent article.

\subsection{Equivalence of ${\rm HRH}(X) \geq 0$ with the condition from \cite{Srinivas-Takagi}}

We now relate the condition ${\rm HRH}(X) \geq 0$, in the case of isolated normal singularities, to the Hodge theoretic condition in \cite{Srinivas-Takagi}, which proves that our Theorem \ref{thm-main} is a generalization of that in \emph{loc.\ cit.} to the non-isolated singularities case.

\begin{prop} Let $X$ be a normal $d$-dimensional variety over $\C$ with an isolated singularity at $x\in X$ and let $\pi \colon (\widetilde{X},E) \to (X,x)$ be a strong log resolution of singularities, with $E$ a simple normal crossings divisor. Then the following statements are equivalent:
\begin{enumerate}
\setlength\itemsep{0.2em}
    \item ${\rm HRH}(X) \geq 0$,
    \item ${\rm Gr}_F^0 H^i_{\{x\}}(X,\C) = 0$ for all $i\in \Z$,
    \item $H^i(E,\cO_E) = 0$ for all $i\geq 1$.
\end{enumerate}
\end{prop}
\begin{proof} Recall that \cite{DOR1}*{Theorem D(2)} shows that ${\rm HRH}(X) \geq 0$ if and only if
\[ F^{d} H^i_{\{x\}}(X,\C) = \begin{cases} 0 & i < 2d \\ \C & i = 2 d\end{cases}.\]

As $X$ is irreducible at $x$, we know $H^{2d}_{\{x\}}(X)$ is one dimensional \cite{Brion}*{Proposition A1}, so we see that ${\rm HRH}(X) \geq 0$ if and only if ${\rm Gr}_F^d H^i_{\{x\}}(X,\C) = 0$ for all $i  < 2d$.

To prove the claim we will use the link of the singularity and Serre duality on the resolution. 

Let $L_x$ be the link at $x$: by taking an analytic local embedding $X\subseteq \C^N$, we let $S_{\varepsilon}$ be the $(2N-1)$-dimensional sphere centered at $x$ of radius $0 < \varepsilon \ll 1$, and then set $L_x := X \cap S_{\varepsilon}$. The singular cohomology $H^j(L_x)$ carries a mixed Hodge structure, which, by \cite{DurfeeSaito}*{Proposition 3.5} satisfies
\[ H^k(L_x) \cong H^{k+1}_{\{x\}}(X) \text{ as mixed Hodge structures, for all } k > 0,\]
and we have a short exact sequence
\[ 0 \to \Q^H \to H^0(L_x) \to H^1_{\{x\}}(X) \to 0.\]

Since we assume $X$ is normal, hence $S_2$, any punctured neighborhood $X\setminus \{x\}$ is connected, and so the link is connected. Thus, this short exact sequence gives vanishing $H^1_{\{x\}}(X) = 0$. 

If we now define the link invariants 
\[
\ell^{p,q} = \dim_\C {\rm Gr}_F^p H^{p+q}(L_x) \overset{(\star)}{=} \dim_\C {\rm Gr}_F^p H^{p+q+1}_{\{x\}}(X),
\]
where $(\star)$ holds when $p+q>0$, then we have equality \cite{FL3}*{Definition 2.7}:
\[ \ell^{p,q} = \dim_{\C} H^q(E,\Omega_Y^p(\log E)\vert_E),\]
and so Serre duality:
\begin{align*}
H^q(E,\Omega_Y^p(\log E)|_E)^* &\cong H^{d-q-1}(E, \Hom_E\left(\Omega_Y^p(\log E)|_E, \omega_E\right)) \\
&\cong H^{d-q-1}(E, \Hom_Y\left(\Omega_Y^p(\log E), \omega_Y(E)\right)\!|_E)\\
&\cong  H^{d-q-1}(E, \Omega_Y^{d-p}(\log E)|_E)
\end{align*}
gives $\ell^{p,q} = \ell^{d-p,d-q-1}$. Putting this together, we get an equality for all $p+q \geq 1$:
\[ \dim_\C {\rm Gr}_F^{p} H^{p+q+1}_{\{x\}}(X) = \ell^{p, q} = \ell^{d-p, d-q-1} =  \dim_\C {\rm Gr}_F^{d -p} H^{2d -p-q}_{\{x\}}(X),\]
which, by taking $p=0$, gives equality for all $q\geq 1$:
\begin{equation} \label{eq-SerreDualLocCoh} \dim_\C {\rm Gr}_F^{0} H^{q+1}_{\{x\}}(X) = \ell^{0,q} = \ell^{d,d-q-1} = \dim_\C {\rm Gr}_F^{d} H^{2d -q}_{\{x\}}(X).\end{equation}

We can now prove the claimed equivalences. We have already argued that
\[ {\rm HRH}(X) \geq 0 \quad \text{ if and only if } \quad \dim_{\C} {\rm Gr}_F^{d} H^i_{\{x\}}(X,\C) = 0\ \text{ for all } i < 2d,\]
which, by Serre duality \eqref{eq-SerreDualLocCoh} is equivalent to the vanishing
\[\dim_{\C} {\rm Gr}_F^{0} H^i_{\{x\}}(X,\C) = 0\quad \text{ for all }\ i \geq 2.\]

Finally, by comparison with the link invariants, this condition is equivalent to the vanishing $\ell^{0,q} = 0$ for all $q\geq 1$, in other words, to the vanishing
\[ H^q(E,\cO_E) = 0\quad \text{ for all }\ q \geq 1,\]
as claimed.
\end{proof}

\section{Proof of the main theorem}
Before proceeding with the proof, we state and show the following key lemma.

\begin{lem} \label{lem:RationalPerfDefofFNilpt}
Let $X$ be an equidimensional  variety over an $F$-finite field $k$ of positive characteristic $p>0$. Assume that $X$ admits a resolution of singularities $f \colon Y \to X$ which satisfies Grauert-Riemenschneider vanishing:
\begin{equation} \label{eq:GRinLemma}
R^if_*\omega_Y = 0
\end{equation}
for all integers $i>0$.  Moreover assume that $\tau(\omega_X) = \cJ(\omega_X)$. Then $X$ is $F$-nilpotent if and only if the map
\begin{equation} \label{eq:injPerfResInStatement}
H^{i}_x(\cO_{X, {\rm perf}}) \to H^{i}_x(Rf_*\cO_{Y, {\rm perf}})
\end{equation}
is injective for every closed point $x \in X$ and all integers $i \in \Z$.
\end{lem}
\noindent The injectivity of (\ref{eq:injPerfResInStatement}) is equivalent to $\mathbf{F}_p$-rationality from \cite{baudin2026frobeniusstableversiongrauertriemenschneidervanishing}.
\begin{proof}
Consider the following diagram
\begin{equation} \label{eq:perfDiagramEquiv}
\begin{tikzcd}
\cO_{X, {\rm perf}} \ar{r}{f^*} & Rf_*\cO_{Y, {\rm perf}} \\
\cO_X \ar{u} \ar{r}{f^*} & Rf_*\cO_Y \ar{u}.
\end{tikzcd}
\end{equation}
We let $(R,\m)$ be the local $d$-dimensional ring of $X$ at a closed point $x \in X$ and we replace $f \colon Y \to X$ by its localisation at $x$. By Matlis duality, Assumption (\ref{eq:GRinLemma}) implies that
\[
H^i_\m(Rf_*\cO_Y)=0
\]
for all integers $0 \leq i < d$. In particular, $H^i_\m(Rf_*\cO_{Y, {\rm perf}})=0$ for all integers $i<d$. Therefore,
\begin{align*}
\text{ $R$ is weakly $F$-nilpotent } &\iff H^i_\m(R_{\rm perf})=0 \text{ for all $0 \leq i < d$} \\
&\iff H^{i}_x(\cO_{X, {\rm perf}}) \to H^{i}_x(Rf_*\cO_{Y, {\rm perf}}) \text{ is injective for all $0 \leq i < d$.}
\end{align*}

Thus, we can assume that $R$ is weakly $F$-nilpotent and we are left to show that $R$ is $F$-nilpotent if and only if (\ref{eq:injPerfResInStatement}) is injective for $i = d$. Since $\tau(\omega_X) = \cJ(\omega_X)$, we get from (\ref{eq:GRMD}) and (\ref{eq:tauMD}) that
\begin{equation} \label{eq:GR*EqaulInEquiv}
    0^{\rm GR}_{H^d_\m(R)} = 0^{*}_{H^d_\m(R)}. 
\end{equation}
Therefore, we have the following equivalences.
\begin{align*}
\text{$R$ is $F$-nilpotent} &\overset{{\text{Def.\ } \ref{defi:Fnilp}}}{\iff} 0^*_{H^d_\m(R)} \text{ is Frobenius nilpotent} \\
&\ \overset{(\ref{eq:GR*EqaulInEquiv})}{\iff} \  0^{\rm GR}_{H^d_\m(R)} \text{ is Frobenius nilpotent} \\
&\ \overset{(\ref{eq:perfDiagramEquiv})}{\iff} \ f^* \colon H^d_\m(R_{\rm perf}) \to H^d_\m(Rf_*\cO_{Y, {\rm perf}}) \text{ is injective.}
\end{align*}
This concludes the proof.
\end{proof}

\begin{rmk}
With the same assumptions as in Lemma \ref{lem:RationalPerfDefofFNilpt} (namely, the vanishing \eqref{eq:GRinLemma} and the equality $\tau(\omega_X) = \J(\omega_X)$), suppose in addition that $x \in X$ is an isolated singularity with $X \, \backslash\, \{x\}$ smooth, and $f \colon Y \to X$ is a strong log resolution with exceptional divisor $E$.

Then \cite{Srinivas-Takagi} shows that 
\begin{enumerate}
    \item $X$ is $F$-nilpotent $\iff$ $H^i(E,\cO_E)$ is Frobenius nilpotent for $0 < i \leq \dim E$,
    \item $X$ is weakly $F$-nilpotent $\iff$ $H^i(E,\cO_E)$ is Frobenius nilpotent for $0 < i < \dim E$, 
\end{enumerate}
In particular, a characteristic zero singularity is of $F$-nilpotent type if $H^i(E_p,\cO_{E_p})$ is Frobenius nilpotent for $0 < i \leq \dim E$ and $p \gg 0$, and an analogus statement holds for weak $F$-nilpotent-type.

To show (1) and (2), note that the natural map
\[
R^if_*\cO_{Y,{\rm perf}} \to R^if_*\cO_{E,{\rm perf}}
\]
is an isomorphism for all $i \in \Z$ by (\ref{eq:exact-triangle-Eperf}).  Then (1) and (2) follow by Lemma \ref{lem:RationalPerfDefofFNilpt} and the exact triangle
\[
\cO_{X, {\rm perf}} \to Rf_*\cO_{Y, {\rm perf}} \to R^{>0}f_*\cO_{Y, {\rm perf}} \xrightarrow{+1}\.
\]
\end{rmk}

Next, we make some preparatory remarks and set up the notation. 
\begin{setting} \label{setting:char0}
Let $X$ be an equidimensional affine variety of dimension $d$ over a field $k$ of characteristic zero. We fix an embedding $u \colon X \hookrightarrow W$ into a smooth affine variety $W$ such that ${\rm codim}_W(X) \geq 2$. Moreover, we fix a finitely generated $\Z$-subalgebra $A \subseteq k$.  

We will adhere to the following convention: whenever we consider a flat model $Y_A$ over $A$ of a $k$-variety $Y$ we will denote the restriction to a fiber over $s \in S$ by a subscript $Y_s$. 
\end{setting}

\begin{lem} \label{lem-Exceptional} 
With notation as in Setting \ref{setting:char0}, consider a log resolution $\pi \colon T \to W$ of $(W,X)$ fitting inside the following diagram
\begin{equation} \label{eq:blow-up-square}
\begin{tikzcd}
G \ar[hook]{r} \ar{d}{\rho} & T \ar{d}{\pi} \\
X \ar[hook]{r}{u} & W,
\end{tikzcd}
\end{equation}
where $G = \pi^{-1}(X)^{\rm red}$. Consider a flat model of this diagram over $A$.

Then assuming that Conjecture \ref{conj:weakordinarity} holds, there exists a Zariski-dense set of closed points $S \subseteq \Spec A$ such that for all $s \in S$ the natural map
\[ H^i_{x}(R(\rho_s)_* \cO_{G_s}) \to H^i_{x}(R(\rho_s)_* \cO_{G_s,\rm perf})\]
is injective for every integer $i \in \Z$ and every point $x \in X_s$.
\end{lem}
\begin{proof}
First, by replacing $\Spec A$ by a distinguished open subset, we may assume that restriction to every fiber over $\Spec A$ preserves regularity and simple normal crossing property in Diagram (\ref{eq:blow-up-square}). Moreover, by the Grauert-Riemenschneider vanishing, we may assume that
\begin{equation} \label{eq:GRRed}
R^i(\pi_s)_*\omega_{T_s} = 0
\end{equation}
 for every $i > 0$ and $s \in \Spec A$ (cf.\ \cite[Lemma 2.3]{kawakami2026higherfrationalsingularities}). Fix integers $i, e \geq 0$ and consider the following diagram
\[
\begin{tikzcd}
R^i(\pi_{s})_*F^e_*\omega_{T_s}(G_s) \ar{r}{T^e} \ar{d} &  R^i(\pi_{s})_*\omega_{T_s}(G_s) \ar{d} \\  
R^i(\rho_{s})_*F^e_*\omega_{G_s} \ar{r}{T^e} &  R^i(\rho_{s})_*\omega_{G_s}
\end{tikzcd}
\]
in which vertical maps are induced by adjunction and the horizontal arrows are induced by the Frobenius trace map.

By \cite[Theorem 4.4]{BST} we can find a Zariski-dense set of closed points $S \subseteq \Spec A$ such that the top horizontal arrow is surjective for every $s \in S$. Moreover, by (\ref{eq:GRRed}), the vertical arrows are surjective for every $s \in S$. Hence the bottom horizontal arrow:
\[
T^e \colon R^i(\rho_{s})_*F^e_*\omega_{G_s} \to R^i(\rho_{s})_*\omega_{G_s}
\]
is surjective for every $s \in S$ (see also \cite[(4.6.1)]{BST}). Now the statement of the lemma follows by local duality and the Cohen--Macaulayness of the simple normal crossing variety $G_s$.
\end{proof}

In the next few paragraphs, we carefully construct a log resolution $\pi \colon T \to W$ of $(W,X)$ tailored to the proof of the main theorem. First, let $f \colon Y \to X$ be a resolution of singularities of $X$ obtained as a sequence of blow-ups $f \colon Y = Y_n \to Y_{n-1} \to \ldots \to Y_0 = X$ along smooth centers. By blowing-up the same centers on the ambient space $W$ we obtain a projective birational morphism $\phi \colon Z \to W$ fitting inside the following diagram:
\[ \begin{tikzcd} Y \ar[r, hook] \ar[dr, swap, "f"] & E\ar[r, hook] \ar[d, "\tau"] & Z\ar[d, "\phi"]\\ & X \ar[r, "u"] & W\end{tikzcd}\]
where $E := \phi^{-1}(X)^{\rm red}$ is the total inverse image of $X$ which contains $Y$ as a union of irreducible components (cf.\ \cite[Corollary II.7.15]{HartshorneBook}). 

Note that $Z$ is smooth and that $E$ restricted to the open locus $Z \setminus Y \subseteq Z$ is a simple normal crossing divisor. Thus, we can construct a log resolution $\psi \colon T \to Z$ of $(Z,E)$ which is an isomorphism over $Z \setminus Y$. This map fits inside the following diagram
\begin{equation} \label{eq:bigBlowupDiagram} \begin{tikzcd}[column sep = large]  H \ar[r, "v"]  \ar[ddr, bend right = 73, swap, dashed, "h"] \ar[d, "g"] & G \ar[dd, bend left = 30, dashed, "\rho"{pos=0.3}] \ar[r] \ar[d, swap, "\gamma"] & T\ar[d, "\psi"] \ar[dd, bend left = 45, dashed, "\pi"] \\ Y \ar[r] \ar[dr,swap, "f"] & E\ar[r] \ar[d, swap, "\tau"] & Z\ar[d, "\phi"]\\ & X \ar[r, "u"] & W\end{tikzcd}
\end{equation}
where $H := \psi^{-1}(Y)^{\rm red}$, $G := \psi^{-1}(E)^{\rm red}$, and the dashed arrows denote natural compositions.

This induces the following key diagram
\begin{equation} \label{eq:daggerDiagram}
\begin{tikzcd}[column sep = large]
Rh_*\cO_H & \ar{l}{v^*} R\rho_*\cO_G \\
Rf_* \cO_Y \ar{u}{g^*}[swap]{\cong} & \ar{l}{f^*} \DDB^0_X.  \ar{u}{\rho^*}[swap]{\cong}
\end{tikzcd}
\end{equation}
In this diagram:
\begin{enumerate}
    \item the right vertical arrow is an isomorphism by (\ref{eq:Schwede}) as $\pi \colon T \to W$ is a strong log resolution of $(W,X)$ with exceptional locus $G$;
    \item the left vertical arrow is an isomorphism as $g^* \colon \cO_Y \to Rg_*\cO_H$ is an isomorphism by (\ref{eq:Schwede}) given that $Y$ is smooth and $\psi \colon T \to Z$ is a strong log resolution of $(Z,Y)$ with exceptional locus $H$;
    \item the map $f^* \colon \DDB^0_X \to Rf_*\cO_Y$ agrees with the map (\ref{eq:keymap}) from the main theorem.  
\end{enumerate}

The upshot is that
\begin{align} \label{eq:equivOfDagger1}
f^* \colon \DDB^0_X \to Rf_*\cO_Y \text{ is an isomorphism} \! &\iff \! v^* \colon R\rho_*\cO_G \to Rh_*\cO_H \text{ is an isomorphism} \\ \label{eq:equivOfDagger2}
\DDB^0_X \text{ is Cohen--Macaulay} \! &\iff \! R\rho_*\cO_G \text{ is Cohen--Macaulay}.
\end{align}

\begin{proof}[Proof of Theorem \ref{thm-main}]
Consider a flat model of Diagram (\ref{eq:bigBlowupDiagram}) over a finitely generated $\Z$-subalgebra $A \subseteq k$. By replacing $\Spec A$ by a distinguished open subset, we may assume that for every $s \in \Spec A$, the restriction $(-)_s$ preserves regularity and simple normal crossing property in (\ref{eq:bigBlowupDiagram}). Moreover, by Theorem \ref{thm:taureduction} and Grauert-Riemenschneider vanishing, we may assume that (cf.\ \cite[Lemma 2.3]{kawakami2026higherfrationalsingularities})
\begin{align} \label{eq:tauJEqualInProof}
&\tau(\omega_{X_s}) = \cJ(\omega_{X_s}), \text{ and } \\
&R^{>0}(f_s)_*\omega_{Y_s} = 0 \nonumber
\end{align}
for every $s \in \Spec A$.

Next, by (\ref{eq:SchwedeCharP}) and the same argument as in the paragraph below (\ref{eq:daggerDiagram}), we get the following commutative diagram:
\begin{equation} \label{eq:daggerDiagramPerf}
\begin{tikzcd}[column sep = large]
R(h_s)_*\cO_{H_s, {\rm perf}} & \ar{l}[swap]{v_s^*} R(\rho_s)_*\cO_{G_s, {\rm perf}} \\
R(f_s)_* \cO_{Y_s, {\rm perf}} \ar{u}{g_s^*}[swap]{\cong} & \ar{l}{f_s^*} \cO_{X_s, {\rm perf}}.  \ar{u}{\rho_s^*}[swap]{\cong}
\end{tikzcd}
\end{equation}

For the reader’s convenience, we place the complementary diagrams (\ref{eq:daggerDiagram}) and (\ref{eq:daggerDiagramPerf}) side by side to facilitate tracing the main argument:
\begin{equation} \label{eq:daggerDiagramFull}
\begin{tikzcd}
\ar{r}{v^*} R\rho_*\cO_G  & Rh_*\cO_H  \\
\ar{r}[swap]{f^*} \DDB^0_X  \ar{u}{\rho^*}[swap]{\cong} &  Rf_* \cO_Y \ar{u}{g^*}[swap]{\cong} 
\end{tikzcd}
\hspace{1em}
\begin{tikzcd}
R(h_s)_*\cO_{H_s} \ar{r} & R(h_s)_*\cO_{H_s, {\rm perf}} & \ar{l}[swap]{v_s^*} R(\rho_s)_*\cO_{G_s, {\rm perf}} \\
R(f_s)_* \cO_{Y_s} \ar{u}{g_s^*}[swap]{\cong} \ar{r} & R(f_s)_* \cO_{Y_s, {\rm perf}} \ar{u}{g_s^*}[swap]{\cong} & \ar{l}{f_s^*} \cO_{X_s, {\rm perf}}.  \ar{u}{\rho_s^*}[swap]{\cong}
\end{tikzcd}
\end{equation}
Here, the leftmost vertical arrow in the right diagram may be assumed to be an isomorphism by replacing $\Spec A$ by a distinguished open subset (cf.\ \cite[Lemma 2.3]{kawakami2026higherfrationalsingularities}).

We are now ready to prove the theorem. We split the statement into four implications. In what follows we say that a map $g \colon \cF \to \cG$ of elements $\cF, \cG \in D^b_{\rm coh}(X)$ on a Noetherian scheme $V$ is \emph{injective on local cohomology} or \emph{perverse-injective} if the induced map $H^i_x(\cF) \to H^i_x(\cG)$ is injective for all closed points $x \in V$ and all integers $i \in \Z$. \\

\noindent \textbf{Part 1.} $\begin{aligned}[t] f^* \colon \DDB^0_X \xrightarrow{\cong} Rf_*\cO_Y  &\implies \text{$X$ is of open $F$-nilpotent type.} \\ f^* \colon \DDB^0_X \xrightarrow{\cong} Rf_*\cO_Y &\impliedby  \text{$X$ is of open $F$-nilpotent type and Conjecture  \ref{conj:weakordinarity} holds.} \end{aligned}$\\[0.3em]

Up to replacing $\Spec A$ by a distinguished open subset, for a fixed $s \in \Spec A$ we have:
\begin{align*}
f^* \colon \DDB^0_X \xrightarrow{\cong} Rf_*\cO_Y \overset{(\ref{eq:injectivity3})}{\iff}&\  \DDB^0_X \xrightarrow{f^*} Rf_*\cO_Y \text{ is perverse-injective} \\
\overset{(\ref{eq:daggerDiagramFull})}{\iff}&\  R\rho_*\cO_G \xrightarrow{v^*} Rh_*\cO_H \text{ is  perverse-injective } \\
\overset{(\star)}{\iff}&\   R(\rho_s)_* \cO_{G_s} \xrightarrow{v_s^*} R(h_s)_* \cO_{H_s} \text{ is perverse-injective } \\
\overset{(\dagger)}{\implies}&\  R(\rho_s)_* \cO_{G_s, {\rm perf}} \xrightarrow{v_s^*} R(h_s)_* \cO_{H_s, {\rm perf}} \text{ is perverse-injective } \\
\overset{(\ref{eq:daggerDiagramFull})}{\iff}&\  \cO_{X_s, {\rm perf}} \xrightarrow{f_s^*} R(f_s)_* \cO_{Y_s, {\rm perf}} \text{ is perverse-injective} \\[0.3em]
\overset{}{\iff}&\  \text{$X_s$ is $F$-nilpotent},
\end{align*}
where ($\star$) follows by \cite[Proposition 2.11]{kawakami2026higherfrationalsingularities} (the statement assumes normality, but this assumption is unnecessary in its proof) and the last equivalence follows by (\ref{eq:tauJEqualInProof}) and Lemma \ref{lem:RationalPerfDefofFNilpt}. Though not needed for the proof, let us point out for clarity that in all the statements from left to right, we could have replaced \emph{perverse injectivity} by \emph{isomorphism} (cf.\ (\ref{eq:injectivity2})).

This concludes the proof of $(\implies)$. In order to prove $(\impliedby)$, we can invoke Conjecture  \ref{conj:weakordinarity} and Lemma \ref{lem-Exceptional} to find a Zariski-dense set of closed points $S \subseteq \Spec A$ such that
\begin{equation} \label{eq:FromBSTInProof} H^i_{x}(R(\rho_s)_* \cO_{G_s}) \to H^i_{x}(R(\rho_s)_* \cO_{G_s,\rm perf})
\end{equation}
is injective for all $s \in S$, for every integer $i \in \Z$, and for every point $x \in X_s$. In particular, the converse of implication ($\dagger$) holds true, which concludes the proof of $(\impliedby$).

We emphasize that in the proof of $(\impliedby$) it is not enough to assume that $X$ is of dense $F$-nilpotent type, namely that $X_s$ is $F$-nilpotent for $s$ within a dense set of closed points $Z \subseteq \Spec A$. Indeed, for the proof of $(\impliedby)$ to work, we need to pick $s \in S \cap Z$ but this intersection could be empty.

\noindent \textbf{Part 2.} $\begin{aligned}[t] \text{$\DDB^0_X$  is Cohen--Macaulay}  &\implies \text{$X$ is of open weakly $F$-nilpotent type.} \\ \text{$\DDB^0_X$  is Cohen--Macaulay} &\impliedby  \text{$X$ is of open weakly $F$-nilpotent type and Conj.\  \ref{conj:weakordinarity} holds.} \end{aligned}$\\[0.3em]

Up to replacing $\Spec A$ by a distinguished open subset, for $s \in \Spec A$ we have the following:\\[-1em] 
\begin{alignat*}{3}
\text{$\DDB^0_X$ is Cohen--Macaulay} &\iff \text{$R\rho_*\cO_{G}$ is Cohen--Macaulay}  \ &&\text{ by Diagram (\ref{eq:daggerDiagramFull})}  \\[0.3em]
&\iff \text{$R(\rho_s)_*\cO_{G_s}$ is Cohen--Macaulay}  \ &&\text{ by \cite[Proposition 2.11]{kawakami2026higherfrationalsingularities}} \\[-0.2em]
&\,\overset{{\footnotesize (\dagger\dagger)}}{\implies} \text{$R(\rho_s)_*\cO_{G_s, {\rm perf}}$ is Cohen--Macaulay}  \ && \\[0.35em]
&\iff \text{$\cO_{X_s, {\rm perf}}$ is Cohen--Macaulay} \ &&\text{ by Diagram (\ref{eq:daggerDiagramFull})} \\[0.3em]
&\iff \text{$X_s$ is weakly $F$-nilpotent.} \ &&
\end{alignat*}

This concludes the proof of ($\implies$). As for the converse, we can invoke Conjecture  \ref{conj:weakordinarity} and Lemma \ref{lem-Exceptional} to find a Zariski-dense set of closed points $S \subseteq \Spec A$ such that (\ref{eq:FromBSTInProof}) 
is injective for all $s \in S$, for every integer $i \in \Z$, and for every point $x \in X_s$. In particular, the converse of implication ($\dagger\dagger$) holds true, which concludes the proof of $(\impliedby)$.
\end{proof}

We conclude with a collection of examples. First, we describe a useful characterization in the characteristic zero setting.

\begin{rmk} Let $Y$ be a smooth complex algebraic variety and let $f\in \cO_Y(Y)$ be a globally defined regular function. The \emph{Bernstein-Sato polynomial} of $f$ is, by definition, the monic polynomial $b_f(s)$ of least degree such that there exists $P(s) \in \mathcal D_Y[s]$ (here $s$ is a new variable) with equality
\[ b_f(s) f^s = P(s) f^{s+1}.\]
The symbol $f^s$ is formal, where the action of differential operators $P \in \cD_Y$ is induced multiplicatively by the usual power rule:
\[ \theta f^s = \theta(f) s f^{s-1}, \quad \text{ for } \theta \in \cT_Y.\]

This construction is fundamental in the study of complex hypersurface singularities. The roots of this polynomial are always negative rational numbers, and we always have $(s+1) \mid b_f(s)$. Importantly, the variety $X = V(f) \subseteq Y$ is a rational homology manifold if and only if $b_f(s)/(s+1)$ has no integer roots, see \cite{Torrelli}*{Theorem 1.2}.

In the isolated, weighted homogeneous singularities case, we can say more: we even have equality \cite{DOR2}*{Theorem A}
\begin{equation} \label{eq-IsolatedWeightHomog} {\rm HRH}(X) + 2 = \inf \left\{k \mid (s+k) \text{ divides } b_f(s)/(s+1)\right\},\end{equation}
though this equality fails in general. This equality follows because, in the isolated, weighted homogeneous singularity case, the negated Bernstein-Sato roots agree with the spectral numbers (see \cite{JKSYSpectrum}*{Remark 1.5(ii)}).

The polynomial $b_f(s)$ also measures the Du Bois and rational singularities of $X$. If we define the \emph{minimal exponent} $\widetilde{\alpha}(f) = \min\{\lambda \mid (s+\lambda) \mid b_f(s)/(s+1)\}$, then we have
\begin{align*} \widetilde{\alpha}(f) \geq 1 &\iff X \text{ has Du Bois singularities, by \cite{SaitoMHC}, and }\\
\widetilde{\alpha}(f) > 1 &\iff X \text{ has rational singularities, by \cite{SaitoRat}}.
\end{align*}
\end{rmk}

\begin{egs} \label{eg-Examples} We use the examples of \cite{Srinivas-Takagi}*{Example 2.7}. Below let $X = {\rm Spec}(R)$, $X_\C = X\times_\Z \C$ and $X_{\mathbf F_p} = X\times_{\Z} \mathbf F_p$.

\begin{enumerate} \item Let $R = \Z[x,y,z]/(x^2+y^3+z^7)$. Then $X_{\C}$ is a rational homology manifold (for example, the reduced Bernstein-Sato polynomial of $f = x^2+y^3+z^7$ is easily checked to have no integer roots by the Thom-Sebastiani rule \cite{SaitoMicrolocal}*{Theorem 0.8}: its (negated) roots are
\[ \left\{ \frac{41}{42}, \frac{47}{42}, \frac{53}{42}, \frac{55}{42}, \frac{59}{42}, \frac{61}{42}, \frac{65}{42},\frac{67}{42},\frac{71}{42},\frac{73}{42},\frac{79}{42},\frac{85}{42} \right\}.\] 

Thus, ${\rm HRH}(X_\C) = +\infty \geq 0$. In this example, $X_{\mathbf F_p}$ is known to be $F$-nilpotent \cite{Srinivas-Takagi}*{Example 2.7(1)}.

\item Similarly, for $R = \Z[x,y,z]/(x^2+y^3+z^7+xyz)$, we see that if we set $f = x^2+y^3+z^7+xyz$, then the minimal exponent $\widetilde{\alpha}(f)$ is equal to $1$. To see this, Macaulay2 shows that its Bernstein-Sato polynomial is
\[
b_f(s)
= (s+1)^3
  \left(s+\frac{3}{2}\right)
  \left(s+\frac{4}{3}\right)
  \left(s+\frac{5}{3}\right)
  \prod_{j=8}^{13}\left(s+\frac{j}{7}\right).
\]
In particular, $X_\C$ has Du Bois but not rational singularities, so this implies ${\rm HRH}(X_\C) = -1$. Also, $X_{\mathbf F_p}$ is known to be $F$-injective but not $F$-nilpotent \cite{Srinivas-Takagi}*{Example 2.7(2)}, which matches with the characteristic $0$ behavior.

\item For $R = \Z[x,y,z]/(x^4+y^4+z^4)$ we see by \cite{Srinivas-Takagi}*{Example 2.7(3)} that $X_{\mathbf F_p}$ is $F$-nilpotent for $p \not \equiv 1 \pmod 4$. On the other hand, the integer factors of the reduced Bernstein-Sato polynomial of $X_{\mathbf C}$ are $(s+1)$ and $(s+2)$. The first factor, combined with the homogeneity and \eqref{eq-IsolatedWeightHomog}, implies that ${\rm HRH}(X_{\mathbf C}) = -1$. This can also be checked by blowing up the origin, and verifying the non-vanishing of cohomology of the exceptional divisor. 

\item For a non-hypersurface singularity, consider $R = \Z[s^4,s^3t,st^3,t^4]$, which is not even Cohen--Macaulay. Then $X = {\rm Spec}(R_\C)$ satisfies ${\rm HRH}(X_\C) \geq 0$. Indeed, following \cite{MPLocCoh}*{Example 12.7} if we let $\widetilde{R} = \Z[s^4,s^3t,s^2t^2,st^3,t^4]$ be the normalization of $R$ with $\widetilde{X} = {\rm Spec}(\widetilde{R})$, then the normalization map 
\[ \nu \colon \widetilde{X} \to X\]
is a homeomorphism (in the classical topology). On the other hand, $\widetilde{X}$ is the quotient of $\A^2$ by the action of $\mu_4$, hence is a rational homology manifold by \cite{Brion}*{Prop. A.1(3)}. This proves that $X_{\C}$ is also a rational homology manifold, hence ${\rm HRH}(X_{\C}) = +\infty \geq 0$.

Finally, $X_{\mathbf F_p}$ is known to be $F$-nilpotent. To see this, we have the short exact sequence
\[ 0 \to R_{\mathbf F_p} \to \widetilde{R}_{\mathbf F_p} \to \mathbf F_p \cdot (\overline{s^2t^2}) \to 0,\]
which, in view of $\widetilde{R}_{\mathbf F_p}$ being normal, gives isomorphisms 
\begin{align*}H^1_{\frm}(R_{\mathbf F_p}) &\cong \mathbf F_p,\\
H^2_{\frm}(R_{\mathbf F_p}) &\cong H^2_{\frm}(\widetilde{R}_{\mathbf F_p}).
\end{align*}

Note that Frobenius acts by zero on $\mathbf F_p \cdot (\overline{s^2t^2})$: indeed, for any prime $p$, we have $s^{2p} t^{2p} \in R_{\mathbf F_p}$ (obvious for $p=2$ and for $p=2k+1$ odd, we have $s^{4k+2}t^{4k+2} = (st^3)^2(s^4)^k (t^4)^{k-1}$). Thus, $R_{\mathbf F_p}$ is weakly $F$-nilpotent.

Finally, $\widetilde{R}_{\mathbf F_p}$, being a direct summand of the polynomial ring $\mathbf F_p[s,t]$, is strongly $F$-regular \cite{HochsterHunekeFRegular}*{Theorem 5.5(e)}, hence $F$-rational \cite{HochsterHunekeFRegular}*{Theorem 4.2(a)}. 

Now, we claim $0^*_{H^2_{\frm}(R)} = 0$. For this, let $\eta \in 0^*_{H^2_{\frm}(R_{\mathbf F_p})}$. By definition, this means there exists $c\in R\setminus \{0\}$ and $e_0 \in \Z_{\geq 1}$ with $c F^e(\eta) = 0$ for all $e \geq e_0$. If we view this equality in the isomorphic space $H^2_{\frm}(\widetilde{R}_{\mathbf F_p})$, we conclude by Frobenius compatibility and the fact that $R_{\mathbf F_p} \subseteq \widetilde{R}_{\mathbf F_p}$ is an inclusion of domains, that $\eta \in 0^*_{H^2_{\frm}(\widetilde{R}_{\mathbf F_p})} = \{0\}$, by $F$-rationality. Thus, $\eta = 0$ in $H^2_{\frm}(R_{\mathbf F_p})$, too, proving the claim.

\item For an example with possibly non-isolated singularities, let $Y$ be a complex projective variety with rational singularities and let $L$ be an ample line bundle on $Y$. Consider the cone $X = {\rm Spec}(\bigoplus_{m\geq 0} H^0(Y,L^{\otimes m}))$ over $Y$ with conormal bundle $L$; in other words, the contraction of the zero section in the total space of $L^{-1}$. 

Then \cite{CDOIsolated}*{Corollary 7.3} says that ${\rm HRH}(X) \geq 0$ if and only if \[
H^1(\cO_Y) = \ldots = H^{\dim Y}(\cO_Y) = 0.
\]

On the other hand, let $Y$ be a projective variety over a perfect field $k$ of characteristic $p$ which has $F$-rational singularities and with an ample line bundle $L$. By \cite{MaddoxMiller}*{Lemma 5.16}, the cone $X = {\rm Spec}(\bigoplus_{m\geq 0} H^0(Y,L^{\otimes m}))$ has $F$-nilpotent singularities at the origin if and only if the action of Frobenius on $H^q(Y,\cO_Y)$ is nilpotent for all $q\geq 1$.
\end{enumerate}
\end{egs}

\bibliography{bib}

\end{document}